\documentclass[12pt, leqno]{amsart}
\usepackage{graphicx}
\usepackage{amsfonts,delarray,amssymb,amsmath,amsthm,a4,a4wide}
\usepackage{latexsym}
\usepackage{epsfig}
\usepackage{color}
\newtheorem{thm}{Theorem}[section]

\newtheorem{lem}[thm]{Lemma}
\newtheorem{prop}[thm]{Proposition}
\theoremstyle{definition}

\newtheorem{rem}[thm]{Remark}
\numberwithin{equation}{section}
\newcommand{\norm}[1]{\left\Vert#1\right\Vert}
\newcommand{\abs}[1]{\left\vert#1\right\vert}

\newcommand{\R}{\mathbb R}

\newcommand{\eps}{\varepsilon}

\newcommand{\p}{\partial}

\newcommand{\comment}[1]{}
\def\h{\hspace*{.24in}}

\newenvironment{myindentpar}[1]%
{\begin{list}{}%
         {\setlength{\leftmargin}{#1}}%
         \item[]%
}
{\end{list}}
\begin{document}

\title[Boundary regularity for linearized Monge-Amp\`ere equations]{Boundary regularity for solutions to the linearized Monge-Amp\`ere equations}
\author{N. Q. Le}
\address{Department of Mathematics, Columbia University, New York, NY 10027}
\email{\tt  namle@math.columbia.edu}
\author{O. Savin}
\address{Department of Mathematics, Columbia University, New York, NY 10027}
\email{\tt  savin@math.columbia.edu}
%\address{}%
%\email{\tt  }%

%\thanks{The author was partially supported by NSF grant ....}
%\subjclass{}%
%\keywords{}%

%\date{}%
%\dedicatory{}%
%\commby{}%
% ----------------------------------------------------------------
\begin{abstract}
We obtain boundary H\"older gradient estimates and regularity  for solutions to the linearized Monge-Amp\`ere equations under natural assumptions 
on the domain, Monge-Amp\`ere measures and boundary data.  Our results are affine invariant analogues of the boundary H\"older gradient estimates of Krylov. 

\end{abstract}
\maketitle

\section{Introduction}
This paper is concerned with boundary regularity for solutions to the linearized Monge-Amp\`ere equations. The equations we are interested
in are of the form
$$L_u v =g,$$
with
$$L_u v:= \sum_{i, j=1}^{n} U^{ij} v_{ij},$$
where $u$ is a locally uniformly convex function and $U^{ij}$ is the cofactor of the Hessian $D^{2} u.$ The operator $L_u$ appears in
several contexts including affine differential geometry \cite{TW, TW1, TW2, TW3}, complex geometry \cite{D2}, and fluid mechanics \cite{B, CNP, Loe}. As $U = (U^{ij})$ is divergence-free, we can write
$$L_u v = \sum_{i, j}^{n} \p_{i} (U^{ij} D_{j} v)= \sum_{i, j=1}^{n}\p_{i}\p_{j} (U^{ij} v).$$
Because the matrix of cofactors $U$ is positive semi-definite, $L_{u}$ is a linear elliptic partial differential operator, possibly degenerate.

In \cite{CG}, Caffarelli and Guti\'errez developed a Harnack inequality theory for solutions of the homogeneous
equations $L_u v=0$ in terms of the pinching of the Hessian determinant
$$\lambda\leq \det D^{2} u\leq \Lambda.$$
This theory is an affine invariant version of the classical Harnack inequality for uniformly elliptic equations with measurable coefficients.

In this paper, we establish boundary H\"older gradient estimates and regularity for solutions to the linearized Monge-Amp\`ere equations
$L_u v =g$ under natural assumptions on the domain, Monge-Amp\`ere measures and boundary data; see Theorems \ref{main-reg}, \ref{main-reg-gl} and \ref{main-reg-gl2}. These theorems are affine invariant analogues of the boundary H\"older gradient estimates of Krylov \cite{K}. 

The motivation for our estimates comes from the study of convex minimizers $u$ for convex energies $E$ of the type
$$E(u) =\int_{\Omega} F(\det D^{2} u ) \, dx + \int_{\p\Omega} u d\sigma -\int_{\Omega} u dA,$$
which we considered in \cite{LS2}. Such energies appear in the work of Donaldson \cite{D1}-\cite{D4} in the context of existence of K\"{a}hler metrics of constant scalar curvature for toric varieties. Minimizers of $E$ satisfy a system of the form
\begin{equation}\label{EL-intro}
 \left\{
 \begin{alignedat}{2}
   -F'(\det D^{2} u) ~&=v \h~&&\text{in} ~\Omega, \\\
 U^{ij} v_{ij}&= -dA \h~&&\text{in}~ \Omega,\\\
v &=0\h~&&\text{on}~\p \Omega,\\\
U^{\nu \nu} v_{\nu} &= -\sigma~&&\text{on}~\p \Omega,
 \end{alignedat}
 \right.
\end{equation}
where $U^{\nu\nu} = \det D^{2}_{x^{'}} u$ with $x' \perp \nu$ denoting the tangential directions along $\p \Omega$. The minimizer $u$ solves a fourth order elliptic equation with two nonstandard boundary conditions involving the second and third order derivatives of $u$. In \cite{LS2} we apply the boundary H\"older gradient estimates established in this paper and show that $u\in C^{2,\alpha}(\overline{\Omega})$ in dimensions $n=2$ under suitable conditions on the function $F$ and the measures $dA$ and $ d\sigma.$

Our boundary H\"older gradient estimates depend only on the bounds on the Hessian determinant $\det D^{2} u$, the quadratic separations of
$u$ from its tangent planes on the boundary $\p\Omega$ and the geometry of $\Omega$. Under these assumptions, the linearized Monge-Amp\`ere operator $L_u$ is in general not uniformly elliptic, i.e., the eigenvalues of $U = (U^{ij})$ are not necessarily bounded away from $0$ and $\infty.$ Moreover, $L_u$ can be possibly singular near the boundary; even if $\det D^{2} u$ is constant in $\overline{\Omega}$, $U$ can blow up logarithmically at the boundary, see Proposition \ref{u_reg}. The degeneracy and singularity of $L_u$ are the main difficulties in establishing our boundary regularity results. We handle the degeneracy of $L_u$ by working as in \cite{CG} with sections of solutions to the Monge-Amp\`ere equations. These sections have the same role as euclidean balls have in the classical theory. To overcome the singularity of $L_u$ near the boundary, we use a Localization Theorem at the boundary for solutions to the Monge-Amp\`ere equations which was obtained in \cite{S,S2}.  \\

The rest of the paper is organized as follows. We state our main results in Section \ref{results}. In Section 3, we discuss the Localization Theorem and weak Harnack inequality, which are the main tools used in the proof of our local boundary regularity result, Theorem \ref{main-reg}. In Sections \ref{boundary-sec} and \ref{property-sec}, we study boundary behavior and the main properties of the rescaled functions $u_h$ obtained from the Localization Theorem. The proofs of Theorems \ref{main-reg} and \ref{main-reg-gl2} will be given in Section \ref{proof-sec} and Section 7.

\section{Statement of the main results}\label{results}

Let $\Omega\subset \R^{n}$ be a bounded convex set with
\begin{equation}\label{om_ass}
B_\rho(\rho e_n) \subset \, \Omega \, \subset \{x_n \geq 0\} \cap B_{\frac 1\rho},
\end{equation}
for some small $\rho>0$. Assume that 
\begin{equation}
\Omega~ \text{contains an interior ball of radius $\rho$ tangent to}~ \p 
\Omega~ \text{at each point on} ~\p \Omega\cap\ B_\rho.
\label{tang-int}
\end{equation}

Let $u : \overline \Omega \rightarrow \R$, $u \in C^{0,1}(\overline 
\Omega) 
\cap 
C^2(\Omega)$  be a convex function satisfying
\begin{equation}\label{eq_u}
\det D^2u =f, \quad \quad 0 <\lambda \leq f \leq \Lambda \quad \text{in $\Omega$}.
\end{equation}
Throughout, we denote by $U = (U^{ij})$ the matrix of cofactors of the 
Hessian matrix $D^{2}u$, 
i.e., $$U = (\det D^{2} u) (D^{2} u)^{-1}.$$
We assume that on $\p \Omega\cap B_\rho$, 
$u$ separates quadratically from its tangent planes on $\p \Omega$. 
Precisely we assume that if $x_0 \in 
\p \Omega \cap B_\rho$ then
\begin{equation}
 \rho\abs{x-x_{0}}^2 \leq u(x)- u(x_{0})-\nabla u(x_{0}) (x- x_{0}) \leq 
\rho^{-1}\abs{x-x_{0}}^2,
\label{eq_u1}
\end{equation}
for all $x \in \p\Omega.$ 

When $x_0 \in \p \Omega,$ the term $\nabla u(x_0)$ is understood in the 
sense that
$$ x_{n+1}=u(x_0)+\nabla u(x_0) \cdot (x-x_0) $$ is a supporting 
hyperplane for the graph of $u$ but for any $\eps >0$,
$$ x_{n+1}=u(x_0)+(\nabla u(x_0)- \eps \nu_{x_0}) \cdot (x-x_0) $$
is not a supporting hyperplane, where $\nu_{x_0}$ denotes the exterior 
unit normal to $\p \Omega$ at $x_0$. In fact we will show in Proposition 
\ref{U-bound-grad2} 
that our hypotheses imply that $u$ is 
always differentiable at $x_0$ and then $\nabla u(x_0)$ is defined also in 
the classical sense.

\

We are ready to state our main theorem.

\begin{thm}
Assume $u$ and $\Omega$ satisfy the assumptions 
\eqref{om_ass}-\eqref{eq_u1} 
above. Let $v: B_{\rho}\cap 
\overline{\Omega}\rightarrow \R$ be a continuous solution to 
\begin{equation*}
 \left\{
 \begin{alignedat}{2}
   U^{ij}v_{ij} ~& = g \h~&&\text{in} ~ B_{\rho}\cap \Omega, \\\
v &= 0\h~&&\text{on}~\p \Omega \cap B_{\rho},
 \end{alignedat} 
  \right.
\end{equation*} 
Then
\begin{equation*}
 \norm{v_{\nu}}_{C^{0, \alpha} (\p \Omega\cap B_{\rho/2})} 
 \leq C \left (\norm{v}_{L^{\infty}(\Omega\cap B_{\rho})}+
 \|g/ \, tr \, U\|_{L^\infty(\Omega \cap B_\rho)} \right),
\end{equation*}
and, for $r\leq \rho/2$, we have the estimate
\begin{equation*}
 \max_{B_{r}\cap \overline{\Omega}}\abs{v + v_{\nu}(0)x_{n}} \leq C r^{1 
+\alpha} \left (\norm{v}_{L^{\infty}(\Omega\cap B_{\rho})}+
 \|g/ \, tr \, U\|_{L^\infty(\Omega \cap B_\rho)} \right),
\end{equation*}
where $\alpha\in(0, 1)$ and $C$ are constants depending only on $n, \rho,
 \lambda, \Lambda $.
\label{main-reg}
\end{thm}

 We remark that our estimates do not depend on the 
$C^{0,1}(\overline{\Omega})$ norm of $u$ or the smoothness of 
$u$. 

\begin{rem} 
The theorem is still valid if we consider the equation
$$ tr \, (A D^2v)=g, \quad \quad \mbox{ with}   \quad 0<\tilde \lambda U \le A \le \tilde \Lambda U$$
and then the constants $ \alpha$, $C$ depend also on $\tilde \lambda$, $\tilde \Lambda$.
\label{rem-big}
\end{rem}

Theorem \ref{main-reg} is concerned with boundary regularity in the case when the 
potential $u$ is nondegenerate along $\p \Omega$. It is an affine 
invariant 
analogue of the boundary H\"older gradient estimate of Krylov \cite{K}.  

\begin{thm}[Krylov]
 Let $w\in C(\overline{B_1^+}) \cap C^2(B_1^+)$ 
satisfy $$L w=f \quad \mbox{in $B_1^+$}, \quad \quad w=0 \quad \mbox{on 
$\{x_n=0\}$},$$
where $L=a^{ij} \p_{ij}$ is a uniformly elliptic operator with bounded 
measurable 
coefficients with ellipticity constants $\lambda, \Lambda.$ Then
there are constants $0<\alpha<1$ and $C>0$ depending on $\lambda$, 
$\Lambda$, $n$ such that
\begin{equation*}
 \norm{ w_n}_{C^{\alpha}(B_{1/2} \cap \{x_n=0\})}\leq 
C(\|w\|_{L^\infty(B_1^+)} + \|f\|_{L^\infty(B_1^+)}).
\end{equation*}  
\end{thm}

 We also obtain global boundary regularity estimates 
under global conditions on the domain $\Omega$ and the
potential function $u$.

\begin{thm}
Assume that 
$\Omega \subset B_{1/\rho}$ contains an interior ball of radius $\rho$ 
tangent to $\p 
\Omega$ at each point on $\p \Omega.$
Assume further that $$ \det D^2 u =f \quad \quad \mbox{with} \quad \lambda \le f \le \Lambda,$$ and on $\p \Omega$, $u$ 
separates quadratically from its 
tangent planes,  namely
\begin{equation*}
 \rho\abs{x-x_{0}}^2 \leq u(x)- u(x_{0})-\nabla u(x_{0}) (x- x_{0})
 \leq \rho^{-1}\abs{x-x_{0}}^2, ~\forall x, x_{0}\in\p\Omega.
\end{equation*}
Let $v: \overline{\Omega}\rightarrow \R$ be a continuous function that 
solves  
\begin{equation*}
 \left\{
 \begin{alignedat}{2}
   U^{ij}v_{ij} ~& = g \h~&&\text{in} ~ \Omega, \\\
v &= \varphi \h~&&\text{on}~\p \Omega,
 \end{alignedat} 
  \right.
\end{equation*}
where $\varphi$ is a $C^{1,1}$ function defined on $\p\Omega$. 
Then
\begin{equation*}
 \norm{v_{\nu}}_{C^{0, \alpha} (\p \Omega )} \leq C\left(\|\varphi\|_{C^{1,1}(\p \Omega)} +\|g/ \, tr \, U\|_{L^\infty(\Omega)} \right), 
\end{equation*}
and for all $x_0 \in \p \Omega$
\begin{equation*}
 \max_{B_{r}(x_{0})\cap \overline \Omega}\abs{v- v(x_{0})- \nabla v(x_{0})(x - 
x_{0})} \leq C \left (
\|\varphi\|_{C^{1,1}(\p \Omega)} +\|g/ \, tr \, U\|_{L^\infty(\Omega)} \right) r^{1 +\alpha}, 
\end{equation*}
where $\alpha\in(0, 1)$ and $C$ are constants depending on 
$n, \rho, \lambda, \Lambda$.\\
\label{main-reg-gl}
\end{thm}

Theorem \ref{main-reg-gl} follows easily from Theorem \ref{main-reg}. Indeed, first we notice that $v$ is bounded by the use of barriers $$\pm C(|x|^2-2/\rho^2),$$ for appropriate $C$, and then we  apply Theorem \ref{main-reg} on $\p \Omega$ for $\tilde v:=v- \varphi$, where $\varphi$ is a $C^{1,1}$ extension of $\varphi$ to $\overline \Omega$. 

If, in addition, we assume that $\det D^{2} u$ is globally H\"older continuous, then the solutions to the linearized Monge-Amp\`ere equations have global $C^{1,\alpha}$ estimates as stated in the next theorem.
\begin{thm}
Assume the hypotheses of Theorem \ref{main-reg-gl} hold and $f\in C^{\beta}(\overline{\Omega})$ for some $\beta>0$.
Then
\begin{equation*}
 \norm{v}_{C^{1, \alpha} (\overline \Omega )} \leq K( \|\varphi\|_{C^{1,1}(\p \Omega)}+ \|g/ \, tr \, U\|_{L^\infty(\Omega)}), 
\end{equation*}
with  $K$ a constants depending on 
$n, \beta, \rho, \lambda, \Lambda$ and $\|f\|_{C^{\beta}(\overline{\Omega})}$.
 \label{main-reg-gl2}
\end{thm}

Finally we mention also the regularity properties of the potentials $u$ that satisfy our hypotheses.

\begin{prop}\label{u_reg}
If $u$ satisfies the hypotheses of Theorem \ref{main-reg-gl} then $$[\nabla u]_{C^{\alpha}(\overline \Omega)} \le C.$$ If in addition $f \in C^\beta(\overline \Omega)$ then
$$\|D^2 u\| \le K |\log \eps|^2 \quad \mbox{on} \quad \Omega_\eps = \{x\in\Omega, 
dist(x,\p\Omega)>\eps\},$$
where $K$ is a constant depending on 
$n, \beta, \rho, \lambda, \Lambda$ and $\|f\|_{C^{\beta}(\overline{\Omega})}$.
\end{prop}

The proof of Theorem \ref{main-reg} follows the same lines as the proof of the standard boundary estimate of Krylov. Our main tools are a localization theorem at the boundary for solutions to the Monge-Amp\`ere equation which was obtained in \cite{S}, and the interior Harnack estimates for solutions to the linearized Monge-Amp\`ere equations which were established in \cite{CG} (see Section \ref{tools}).

\section{The Localization Theorem and Weak Harnack Inequality }\label{tools}
In this section, we state the main tools used in the proof of Theorem \ref{main-reg}, the localization theorem and the weak Harnack inequality.\\
\h We start with the localization theorem. 
Let $u:\overline{\Omega}\rightarrow\R$ be a continuous convex function and 
assume that
\begin{equation}\label{0grad}
u(0)=0, \quad \nabla u(0)=0.
\end{equation}
Let $S_{h}(u)$ be the section of $u$ at $0$ with level $h$:
$$S_h := \{x \in \overline \Omega : \quad u(x) < h \}.$$

If the boundary data has quadratic growth near $\{x_n=0\}$ then, as $h \rightarrow 0$, $S_h$ is equivalent to a half-ellipsoid centered at 0. This is the content
of the Localization Theorem proved in \cite{S,S2}. Precisely, this theorem reads as follows.

\begin{thm}[Localization Theorem \cite{S,S2}]\label{main_loc}
 Assume that $\Omega$ satisfies \eqref{om_ass} and $u$ satisfies 
\eqref{eq_u}, 
\eqref{0grad} above and,
\begin{equation}\label{commentstar}\rho |x|^2 \leq u(x) \leq \rho^{-1} 
|x|^2 \quad \text{on $\p \Omega \cap \{x_n \leq \rho\}.$}\end{equation}
Then, for each $h<k$ there exists an ellipsoid $E_h$ of volume $\omega_{n}h^{n/2}$ 
such that
$$kE_h \cap \overline \Omega \, \subset \, S_h \, \subset \, k^{-1}E_h \cap \overline \Omega.$$

Moreover, the ellipsoid $E_h$ is obtained from the ball of radius $h^{1/2}$ by a
linear transformation $A_h^{-1}$ (sliding along the $x_n=0$ plane)
$$A_hE_h= h^{1/2}B_1$$
$$A_h(x) = x - \tau_h x_n, \quad \tau_h = (\tau_1, \tau_2, \ldots, 
\tau_{n-1}, 0), $$
with
$$ |\tau_{h}| \leq k^{-1} |\log h|.$$
The constant $k$ above depends only on $\rho, \lambda, \Lambda, n$.
\end{thm}

 The ellipsoid $E_h$, or equivalently the linear map $A_h$, 
provides useful information about the behavior of $u$ 
near the origin. From Theorem \ref{main_loc} we also control the shape of sections that are tangent to $\p \Omega$ at the origin. Before we state this result we introduce the notation for the section of $u$ centered at $x\in \overline \Omega$ at height $h$:
\begin{equation*}
 S_{x,h} (u) :=\{y\in \overline \Omega:  u(y) < u(x) + \nabla u(x) (y- x) +h\}.
\end{equation*}

\begin{prop}\label{tan_sec}
Let $u$ and $\Omega$ satisfy the hypotheses of the Localization Theorem \ref{main_loc} at the 
origin. Assume that for some $y \in \Omega$ the section $S_{y,h} \subset \Omega$
is tangent to $\p \Omega$ at $0$ for some $h \le c$. Then there exists a small 
 constant $k_0>0$ depending on $\lambda$, $\Lambda$, $\rho $ and $n$ such that
$$ \nabla u(y)=a e_n 
\quad \mbox{for some} \quad   a \in [k_0 h^{1/2}, k_0^{-1} h^{1/2}],$$
$$k_0 E_h \subset S_{y,h} -y\subset k_0^{-1} E_h, \quad \quad k_0 h^{1/2} \le dist(y,\p \Omega) \le k_0^{-1} h^{-1/2}, \quad $$
with $E_h$ the ellipsoid defined in the Localization Theorem \ref{main_loc}.
\end{prop}

Proposition \ref{tan_sec} is a consequence of Theorem \ref{main_loc} and was proved \cite{S3}. For completeness we sketch its proof at the end of the paper. 

Next, we state the weak Harnack inequality. Caffarelli and Guti\'errez \cite{CG} proved H\"older estimates and Harnack inequalities
for solutions of the homogeneous equation $L_{u} v = 0$. Their approach is based on the Krylov and Safonov's  H\"older estimates for
linear elliptic equations in general form, with the sections of $u$ 
having the same role as euclidean balls have in the classical theory. We state the weak Harnack inequality in this setting (see also \cite{TW3}).
\begin{thm} (Theorem 4 \cite{CG})
 Let $u\in C^2(\Omega)$ be a locally strictly convex function satisfying
\begin{equation*}
 0<\lambda \leq \det D^{2} u \leq \Lambda, 
\end{equation*}
and let $v \ge 0$ be a nonnegative supersolution defined in a section 
$S_{x, h} (u)\subset \subset \Omega$,  $$L_u v:= U^{ij} v_{ij} \le 0.$$  If
\begin{equation*}
| \{v\ge 1\} \cap S_{x,h}(u)|\ge \mu |S_{x,h}(u)|
\end{equation*}
then $$ \inf_{S_{x, h / 2}(u)} v \ge c,$$ 
with $c>0$ a constant depending only on $n$, $\lambda$, $\Lambda$ and $\mu$.
\label{osci-thm}
\end{thm}

\section{Boundary behavior of the rescaled functions}\label{boundary-sec}

We denote by $c$, $C$ positive constants depending on $\rho$, $\lambda$, $\Lambda$, 
$n$, and their values may change from line to line whenever 
there is no possibility of confusion. We refer to such constants as {\it universal constants}.

Sometimes, for simplicity of notation, we write  $S_{x,h}$ instead of $S_{x,h}(u)$ and we drop the $x$ subindex whenever $x=0$, i.e., $S_h=S_{0,h}(u)$.

We denote the distance from a point $x$ to a closed set $\Gamma$ as
\begin{equation}\label{dist}
d_\Gamma(x)=dist(x, \Gamma).
\end{equation}

First we obtain pointwise $C^{1, \alpha}$ estimates on the boundary in the setting of
the Localization Theorem \ref{main_loc}. We know that for all $h \le k$,
 $S_h$ satisfies 
$$k E_h \cap \bar \Omega  \subset S_h \subset k^{-1} E_h,$$ 
with $A_h$ being a linear transformation 
and
$$\det A_{h} = 1,~E_h=A^{-1}_hB_{h^{1/2}}, \quad \quad A_hx=x-\tau_hx_n$$
$$\tau_h \cdot e_n=0, \quad \|A_h^{-1}\|, \,\|A_h\| \le k^{-1} |\log h|.$$
This gives 
\begin{equation}\label{small-sec}
 \overline \Omega \cap B^{+}_{ch^{1/2}/\abs{\log h}}\subset S_{h}  
\subset B^{+}_{C h^{1/2} \abs{\log h}},
\end{equation}
or $$|u| \le h \quad \mbox{in} \quad 
\overline \Omega \cap B^{+}_{ch^{1/2}/\abs{\log h}}
.$$
Then for all $x$ close to the origin $$|u(x)| \le C |x|^2|\log x|^2,$$ 
which shows that $u$ is differentiable 
at $0$. We remark that the other inclusion of \eqref{small-sec} gives a lower bound for $u$ near the origin
\begin{equation}\label{lbd}
u(x) \ge c |x|^2 |\log x|^{-2} \ge |x|^3.
\end{equation}

We summarize the differentiability of $u$ in the next lemma.

\begin{lem}
Assume $u$ and $\Omega$ satisfy the hypotheses of the Localization Theorem \ref{main_loc} at a point $x_0 \in \p \Omega$. If $x\in \overline \Omega \cap B_r(x_0)$, $r \le 1/2$, then
\begin{equation}
 \abs{u(x)-u(x_{0})-\nabla u(x_{0})(x-x_{0})} \leq C r^{2} |\log r|^2. 
\label{U-bound-grad}
\end{equation}
Moreover, if $u$, $\Omega$ satisfy the hypotheses of the Localization Theorem \ref{main_loc} also at a point $x_1 \in \p \Omega \cap B_r(x_0)$ then
$$|\nabla u(x_1)-\nabla u(x_0)| \le C r |\log r|^2.$$
\label{U-bound-grad2}
\end{lem}

Clearly, the second statement follows from writing \eqref{U-bound-grad} for $x_0$ and $x_1$ at all points $x$ in a ball $B_{cr}(y) \subset \Omega$.

Next we discuss the scaling for our linearized Monge-Amp\`ere equation. Under the linear transformations
\begin{equation*}
 \tilde{u}(x)= \frac{1}{a}u(Tx), \quad \quad \tilde{v}(x) = \frac{1}{b}v(Tx), 
\end{equation*}
$$\tilde{g}(x) = \frac{1}{a^{n-1}b}(\det T)^2 g(Tx),  $$
we find that
\begin{equation}
 \tilde{U}^{ij}\tilde{v}_{ij} = \tilde{g}. 
\label{eq-trans}
\end{equation}
Indeed, we note that
\begin{eqnarray*}
 D^{2}\tilde{u}= \frac{1}{a}T^{t} D^{2} u T, \quad D^{2}\tilde{v}= 
\frac{1}{b}T^{t} D^{2} v T,
\end{eqnarray*}
and
\begin{eqnarray*}
 \tilde{U} & =& (\det  D^{2}\tilde{u}) ( D^{2}\tilde{u})^{-1} \\ 
&=&\frac{1}{a^{n-1}}(\det T)^2 (\det D^{2} u) \, T^{-1} (D^{2} u)^{-1} 
(T^{-1})^{t}\\
&=& \frac{1}{a^{n-1}}(\det T)^2 T^{-1} U (T^{-1})^{t}
\end{eqnarray*}
and (\ref{eq-trans}) easily follows.\\
We use the rescaling above with $$a = h, \quad b = h^{1/2}, \quad ~T = h^{1/2} 
A_h^{-1}$$
where $ A_{h}$ is the matrix in the Localization theorem. 
We denote the rescaled functions by \begin{equation*}
 u_h(x):=\frac{u(h^{1/2}A^{-1}_hx)}{h}, \quad \quad  
v_h(x):=\frac{v(h^{1/2}A^{-1}_hx)}{h^{1/2}}.
\end{equation*}
$$g_h(x):=h^{1/2}g(h^{1/2}A^{-1}_hx),$$ 
and they satisfy
\begin{equation}
U^{ij}_{h} D_{ij} v_{h} = g_h.
\label{sol-trans}
\end{equation}

The function $u_h$ is continuous and is defined in $\overline \Omega_h$ with
 $$\Omega_h:= h^{-1/2}A_h \Omega,$$
 and solves the Monge-Amp\`ere equation
 $$\det D^2 u_h=f_h(x), \quad \quad \lambda \le f_h \le \Lambda,$$
 with
 $$f_h(x):=f(h^{1/2}A_h^{-1}x).$$
 The section at height 1 for $u_h$ centered at the origin satisfies
$$S_1(u_h)=h^{-1/2}A_hS_h,$$ and by the localization theorem we obtain $$B_k \cap \overline \Omega_h \subset S_1(u_h) \subset B_{k^{-1}}^+.$$

 We remark that since $$tr \, U =h^{-1} \, tr (T U_h T^t) \le h^{-1} \|T\|^2 \, tr 
\, U_h,$$ we obtain
\begin{equation}\label{gh}
\|g_h/ \,tr \, U_h\|_{L^\infty} \le C  h^{1/2} |\log h|^2 \, \,  \|g/ \, tr \, U\|_{L^\infty} .
\end{equation}

In the next lemma we investigate the properties of the rescaled function $u_h$.
We recall that if $x_0 \in \p \Omega \cap B_\rho$ then $\Omega$ has an interior tangent ball of radius $\rho$ at $x_{0}$, and $u$ satisfies
\begin{equation}
 \rho\abs{x-x_{0}}^2 \leq u(x)- u(x_{0})-\nabla u(x_{0}) (x- x_{0}) \leq \rho^{-1}\abs{x-x_{0}}^2,~\quad \forall x \in \p\Omega.
\label{sep-near01}
\end{equation}

\begin{lem}
If $h\leq c$, then 

a) $\p \Omega_h \cap B_{2/k}$ is a graph 
in the $e_n$ direction whose $C^{1,1}$ norm 
is bounded by $C h^{1/2}$;

b) for any 
$x,x_{0}\in\p\Omega_{h}\cap B_{2/k}$ we have
\begin{equation}
 \frac{\rho}{4}\abs{x-x_{0}}^2 \leq u_h(x) - u_h(x_{0}) -\nabla u_h(x_{0}) (x- x_{0})\leq 4\rho^{-1} \abs{x- x_{0}}^2, 
\label{sep-near0}
\end{equation}

c) if $r \le c$ small, we have $$|\nabla u_h| \le 
C r |\log r|^2 \quad \mbox{in} \quad \overline \Omega_h \cap B_r.$$

\label{sep-lem}
\end{lem}

\begin{proof}
For  $x, x_{0}\in\p\Omega_{h}\cap B_{2/k}$ we denote 
\begin{equation*}
 X =Tx, \quad X_0=Tx_0, \quad T:= h^{1/2} A^{-1}_{h},
\end{equation*}
hence $$ X, X_0 \in \p \Omega \cap B_{Ch^{1/2}\abs{\log h}}.$$
First we show that
\begin{equation}
\frac{\abs{x-x_0}}{2}\leq  \frac{\abs{X-X_{0}}}{h^{1/2}} \leq 2 \abs{x-x_{0}},
\label{sep-near02}
\end{equation}
which is equivalent to
$$1/2 \le |A_hZ|/|Z| \le 2, \quad \quad Z:=X-X_0.$$
Since $\p \Omega$ is $C^{1,1}$ in a neighborhood of the origin we find
$$|Z_n| \le Ch^{1/2}|\log h||Z'|$$ hence, if $h$ is small
$$|A_hZ-Z|=|\tau_hZ_n| \le Ch^{1/2}|\log h|^2 |Z'| \le |Z|/2,$$ 
and \eqref{sep-near02} is proved.

Part b) follows now from (\ref{sep-near01}) and the equality 
\begin{equation*}
 u_h(x) - u_h(x_{0}) -\nabla u_h(x_{0}) (x- x_{0}) 
=\frac{1}{h} (u(X)-u(X_0)-\nabla u(X_0)(X-X_0).
\end{equation*}

Next we show that $\p \Omega_h$ has small $C^{1,1}$ norm. Since $\p \Omega$ has an 
interior tangent ball at $X_0$ we see that
$$|(X-X_0) \cdot \nu_0| \le C|X-X_0|^2,$$
where $\nu_0$ is the exterior normal to $\Omega$ at $X_0$. This implies, in view of (\ref{sep-near02})
$$|(x-x_0)\cdot T^t \nu_0| \le C h |x-x_0|^2,$$
or
$$|(x-x_0) \cdot \tilde \nu_0|\le C\frac{h}{|T^t \nu_0|}|x-x_0|^2,$$ where 
$$\tilde \nu_0:=T^t \nu_0 /|T^t \nu_0|.$$
From the formula for $A_h$ we see that
$$e_n \cdot ((A_h^{-1})^Te_n)=(A_h^{-1}e_n)\cdot e_n=1,$$
hence $$|(A_h^{-1})^Te_n| \ge 1. $$ Since $$|\nu_0 + e_n| \le Ch^{1/2} |\log h|$$
we obtain $$|(A_h^{-1})^T \nu_0| \ge 1-C h^{1/2}\, |\log h| \, \|A_h^{-1}\| \ge 1/2 ,$$
thus 
$$|T^t \nu_0|=h^{1/2} |(A_h^{-1})^T\nu_0| \ge h^{1/2}/2.$$
In conclusion
$$|(x-x_0) \cdot \tilde \nu_0|\le C h^{1/2} |x-x_0|^2,$$
which easily implies our claim about the $C^{1,1}$ norm of $\p \Omega_h$.

Next we prove property c). From a), b) above we see that $u_h$ 
satisfies in $S_1(u_h)$ the hypotheses of the Localization Theorem 
\ref{main_loc} at $0$ for a small 
$\tilde \rho$ depending on the given constants. We consider a point $x_0 \in \p 
\Omega_h 
\cap B_r$, and by Lemma \ref{U-bound-grad2}, it remains to show that $u_h$, $S_1(u_h)$ satisfy the hypotheses of the Localization Theorem \ref{main_loc} also at $x_0$.  From \eqref{U-bound-grad} we have
\begin{equation}\label{ubd}
|u_{h}| \le Cr^2 |\log r|^2 \quad \mbox{in} \quad  \overline \Omega_h \cap B_{2r},
\end{equation} 
which, by convexity of $u_h$ gives $$\partial_n u_h(x_0) \le C r |\log r|^2.$$
On the other hand, we use part b) at $x_0$ and $0$ (see \eqref{sep-near0}) and obtain
$$ |u_h(x_0)+\nabla u_h(x_0) \,  (x-x_0)| \le C r^2 \quad \mbox{on} \quad \p \Omega \cap B_r,$$ 
thus
$$|\nabla u_h(x_0) \cdot x| \le C r^2 \quad \mbox{on} \quad \p \Omega \cap B_r.$$
Since $x_n \ge 0$ on $\p \Omega$, we see that if $\partial_n u_h(x_0)\ge 0$ then,
$$ \nabla_{x'} u_h(x_0) \cdot x' \le C r^2 \quad \mbox{if} \quad |x'| \le r/2,$$ which 
gives 
$$|\nabla_{x'} u_h(x_0)| \le Cr.$$
We obtain the same conclusion similarly if $\partial_n u_h(x_0)\le 0$. The upper bounds 
on $\partial_n u_h(x_0)$ and $|\nabla_{x'}u_h(x_0)|$ imply that
if $x \in S_1(u_h) \subset B_{1/k}$ we have
 $$u_h(x_0)+\nabla u_{h}(x_0) \cdot (x-x_0) + 1/2 \le Cr^2 + Cr|\log r|^2 + 1/2 <1,$$
provided that $r$ is small.  
This shows that $$S_{x_0, \frac 12}(u_h) \subset S_1(u_h) \subset B_{1/k}.$$ Moreover, 
since $$|S_{x_0,\frac 12}(u_h)|=h^{-n/2}|S_{X_0, \frac h 2}(u)| \sim 1$$ we obtain a bound for $|\nabla u(x_0)|$. Now we can 
easily conclude from parts a) and b) and the inclusion above that
$u_h$ satisfies in $S_1(u_h)$ the hypotheses of the Localization Theorem 
at $x_0$ for a small $\tilde \rho$.

\end{proof}

In the next proposition we compare the distance functions under the following transformations of point and domain:
\begin{equation*}
 x\rightarrow X:=Tx, ~\Omega_h \rightarrow \Omega= T \Omega_h, \quad \quad T=h^{1/2}A_h^{-1}.
\end{equation*}

\begin{prop}\label{bdr-distort}
For $x\in \Omega_{h}\cap B^{+}_{k^{-1}}$, let $X= Tx \in\Omega.$ Then (see notation \eqref{dist})
\begin{equation*}
1 - C h^{1/2}\abs{\log h}^{2}\leq \frac{h^{-1/2} d_{\p \Omega} (X)}{d_{\p \Omega_h}(x)} \leq 1 + C h^{1/2}\abs{\log h}^{2}.
\end{equation*}
\end{prop}

\begin{proof}
Denote by $\xi_x$, $\xi_X$ the unit vectors at $x$, $X$ which give the perpendicular direction to $\p \Omega_h$ respectively $\p \Omega$, and which point inside the domain. Since
$\p \Omega$ is $C^{1,1}$ at the origin and $|X| \le C h^{1/2} |\log h|$ we find
$$|\xi_X-e_n| \le C h^{1/2} |\log h|.$$
Moreover, the $C^{1,1}$ bound of $\p \Omega_h$ from Lemma \ref{sep-lem} shows that
$$|\xi_x-e_n|\le C h^{1/2}.$$
We compare $h^{-1/2}d_{\p \Omega}(X)$ with $d_{\p \Omega_h}(x)$ by computing the directional derivative of $ h^{-1/2}d_{\p \Omega}(X)$ along $\xi_x$. We have
\begin{align*}
\nabla_x (h^{-1/2} d_{\p \Omega}(X)) \cdot \xi_x&=h^{-1/2} \nabla_X d_{\p \Omega}(X) \, \, T   \, \xi_x \\
&=h^{-1/2}\xi_X \cdot (T \xi_x) \\
&=\xi_X \cdot (A_h^{-1} \xi_x)
\end{align*}
From the inequalities above on $\xi_x$, $\xi_X$ we find $$|\xi_X \cdot (A_h^{-1} \xi_x)-e_n \cdot (A_h^{-1} e_n)|\le Ch^{1/2} |\log h|\|A_h^{-1}\| \le C h^{1/2} |\log h|^2.$$
Using  $$e_n \cdot (A_h^{-1} e_n)=1$$ we obtain
$$|\nabla_x (h^{-1/2} d_{\p \Omega}(X)) \cdot \xi_x-1| \le Ch^{1/2}|\log h|^2,$$
which implies our result.

\end{proof}

\section{The class $\mathcal D_{\sigma}$ and its main properties}\label{property-sec}

In this section we introduce the class $\mathcal D_\sigma$ that captures the properties of the rescaled functions $u_h$ in $S_1(u_h)$. By abuse of notation we use $u$ and $\Omega$ when we define $\mathcal D_\sigma$.

Fix $\rho, \lambda, \Lambda$. We introduce the class $\mathcal D_{\sigma}$ consisting of pairs of function $u$ and domain $\Omega$ satisfying the
following conditions:
\begin{myindentpar}{1cm}
(i) $0 \in \p\Omega, \quad  \Omega \subset B_{1/k}^+, \quad |\Omega| \ge c_0$,\\
(ii) $u:\overline{\Omega}\rightarrow \R$ is convex, continuous satisfying 
$$u(0)=0, \quad \nabla u(0)=0, \quad \lambda \leq \det D^{2} u \leq \Lambda;$$\\
(iii) $$\p \Omega \cap \{u<1\} \subset G \subset  \{x_n \le \sigma \}$$
 where $G$ is a graph in the $e_n$ direction which is defined in $B_{2/k}$, and its $C^{1,1}$ norm is bounded by $\sigma$. \\ 
(iv) $$\frac \rho 4\abs{x-x_{0}}^2 \leq u(x)- u(x_{0})-\nabla u(x_{0}) (x- x_{0}) \leq \frac 4\rho \abs{x-x_{0}}^2 \quad \quad \forall x, x_{0}\in G\cap \p \Omega;$$\\
(v) If $r \le c_0$, 
$$|\nabla u| \le C_0 r|\log r|^2 \quad \quad \mbox{in} \quad \overline \Omega \cap B_r.$$
\end{myindentpar}
The constants $k$, $c_0$, $C_0$ above depend explicitly on $\rho$, $\lambda$, $\Lambda$,  and $n$.

We remark that the properties above imply that if $x_0 \in \p \Omega$ is close to the origin then
$$S_{x_0, \frac 1 2}(u) \subset \{u<1\},$$
and $u$ satisfies in $S_{x_0,\frac 12}(u)$ the hypotheses of the Localization Theorem at $x_0$ for some $\tilde \rho$ depending on the given constants. 

Lemma \ref{sep-lem} can be restated in the following way.
\begin{lem}
 Let $(u, \Omega)$ be as in Theorem \ref{main-reg}. Then, if $h \le c$, $$(u_{h},S_1(u_h))\in \mathcal D_{\sigma} \quad \mbox{with} ~ \sigma=C h^{1/2}.$$
\end{lem}

We first construct a useful subsolution.

\begin{lem}[Subsolution]
 Suppose $(u,\Omega)\in \mathcal D_{\delta}$. If $\delta \le c$ then the function
$$\underline w:= x_{n} -  u + \delta^\frac {1}{n-1} |x'|^2 + 
\frac{\Lambda^n}{\lambda^{n-1}\delta} x_{n}^2$$ satisfies 
 $$L_{u} (\underline w):=U^{ij} \underline w_{ij} \geq \delta^\frac{1}{n-1} \, \, tr \, U,$$
and on the boundary of the domain $D:= \{x_{n}\leq 2 \delta \}\cap \Omega$ we have
$$\underline w \leq 0~\text{on}~\p D\backslash  F_{\delta}, 
\quad \underline w \leq 1 ~ \mbox{on} ~  F_{\delta},$$ 
where 
\begin{equation}\label{csigma}
 F_\delta:= \{x_n=2 \delta, \quad |x'| \le \delta^\frac{1}{6(n-1)}\}.
\end{equation}

\label{key-lem}
\end{lem}

 \begin{proof}

Let
\begin{equation*}
 p(x) =\frac{1}{2} \left( \delta^{\frac{1}{n-1}} |x'|^2 + \frac{\Lambda^n}{\lambda^{n-1}\delta} x_{n}^2\right).
\end{equation*}
Then $$\det D^{2} p(x) = \frac{\Lambda^n}{\lambda^{n-1}}.$$
Using the matrix inequality
\begin{equation*}
 tr (AB) \geq n (\det A \det B)^{1/n}~\text{for}~ A, B~\text{symmetric}~\geq 0,
\end{equation*}
we get
\begin{equation*}
 L_u p=U^{ij} p_{ij} \geq n (\det (U) \det D^{2} p)^{1/n} = n ((\det D^{2}u)^{n-1}\frac{\Lambda^n}{\lambda^{n-1}})^{1/n}\geq n\Lambda.
\end{equation*}
Since $\delta$ is small $$D^2 p \ge \delta^\frac{1}{n-1}I,$$ hence
$$L_u p=U^{ij} p_{ij} \ge \delta^\frac{1}{n-1} \, \, tr \,U.$$
Using $L_u x_n=0$ and $$L_uu=U^{ij}u_{ij}=n \det D^2 u \le n \Lambda$$ we find
$$L_u \underline w =L_u(x_n- u + 2 p) \ge  \delta^\frac{1}{n-1} \, \, tr \, U.$$

Next we check the behavior of $\underline w$ on $\p D$. We decompose $\p D \subset G \cup E_{\delta} \cup F_\delta$ where
$$E_{\delta} := \p D \cap \{ |x| \ge \delta ^\frac{1}{6(n-1)}\}.$$

On $G \cap \p \Omega $, we use the properties of $\mathcal D_\delta$ and obtain $$u \ge (\rho/4)|x|^2, \quad \quad x_n \le \delta |x'|^2$$ which follows from the $C^{1,1}$ bound on the graph $G$.
Then
\begin{equation*}
 \underline w \leq (\delta + \delta^\frac{1}{n-1} + C \delta)|x'|^2- (\rho/4)|x|^2 \leq 0,
\end{equation*}
provided that $\delta$ is small.

On $E_\delta$ we use \eqref{lbd} and find
$$u \ge (\delta^\frac{1}{6(n-1)})^3=\delta^\frac{1}{2(n-1)}$$
hence, for small $\delta$,
$$w \le C\delta- \delta^\frac{1}{2(n-1)} + 
c \delta^\frac {1}{n-1} \le - \delta^\frac{1}{2(n-1)}/2< 0.$$

On $F_{\delta}$, the positive terms in $\underline w$ are bounded by $1/3$ for 
small $\delta$ and we obtain $w \le 1$.

\end{proof}

\begin{rem}
For any point $x_0\in \p \Omega$ close to the origin we can construct the corresponding subsolution
$$\underline w_{x_0}=z_n-u_{x_0} + 2 p (z) $$
where $$u_{x_0}:=u-u(x_0)-\nabla u(x_0) \cdot (x-x_0) $$
and with $z$ denoting the coordinates of the point $x$ in a system of coordinates centered at  $x_0$ with the $z_n$-axis perpendicular to $\p \Omega$. From the proof above we see that $\underline w_{x_0}$ satisfies the same conclusion of Lemma \ref{key-lem} if  $|x_0| \ll \delta$.   
\end{rem}

Next we show that $u$ has uniform modulus of convexity on the set $F_\delta$ introduced above (see \eqref{csigma}). 

\begin{lem} Let $(u, \Omega)\in \mathcal D_{\delta}$. If $\delta \le c$ then for any $y \in F_\delta$ 
we have $$S_{y,c \delta^2}(u) \subset \Omega.$$ 
\label{uni-convex}
\end{lem}  

\begin{rem} From now on we fix the value of $\delta$ to be small, universal so that it satisfies the 
hypotheses of Lemma \ref{key-lem} and Lemma \ref{uni-convex}.
\end{rem}

\begin{rem} \label{bardelta}
Since the section $S_{y,c \delta^2/2}$ is contained in $\Omega \subset B_{1/k}$ and has volume bounded from below we can conclude that it contains a ball $B_{\bar \delta}(y)$ for some $\bar \delta \ll \delta$ small, universal.
\end{rem}

We sketch the proof of Lemma \ref{uni-convex} below.

\begin{proof}
Let $h_0$ be the maximal value of $h$ for which $S_{y,h} \subset \Omega,$ and let
$$x_0 \in \p S_{y,h_0} \cap \p \Omega.$$ 
Since $S_{y,h_0}$ is balanced around $y$ and $u$ grows quadratically away from $0$ on $G$ 
we see that the point $x_0$ lies also in a neighborhood of the origin. Now we can apply Proposition \ref{tan_sec} at $x_0$ and obtain $$h_0 \ge c d_{\p \Omega}(y)^2 \ge c \delta^2.$$ 

\end{proof}

A consequence of Lemma \ref{key-lem} is the following proposition.

\begin{prop}
Assume $(u,\Omega) \in \mathcal D_\delta$, and let $v \ge 0$ be  a nonnegative function satisfying
 $$  L_u v \leq \delta^\frac{1}{n-1}\,  tr \, U ~\text{in} ~\Omega,
\quad \quad v \geq 1~\text{in}~  F_\delta.$$
Then, 
$$ v \geq \frac 12 d_ G~\text{in}~ S_\theta.$$
for some small $\theta$ universal.
\label{key-thm}
\end{prop}

\begin{proof}
Lemma \ref{key-lem} and the maximum principle for the operator $L_u$ imply
$v \ge \underline w$ in $D$,
which gives $$v(0,x_n) \ge \frac 12 x_n \quad \mbox{for} ~  x_n \in [0,c].$$ 
The same argument can be repeated at points $x_0 \in \p \Omega$ if $|x_0|$ is sufficiently small, by comparing $u$ with the corresponding subsolution $\underline w_{x_0}$. We obtain
$$v \ge \frac 12 d_G \quad \mbox{in} \quad \overline \Omega \cap B_c,$$
and the lemma follows by choosing $\theta$ sufficiently small.
\end{proof}

\begin{prop}\label{iteration}
Let $(u,\Omega) \in \mathcal D_\sigma$, ($\sigma \le \delta$) 
and suppose $v$ satisfies in $\Omega$
$$L_u v=g, \quad \quad a \, d_G \le v \le b \, d_G,$$
for some $a,b \in [-1,1]$. There exists $c_1$ small, universal such that if
$$\max \{ \sigma,   \|g/ tr \, U \|_{L^\infty}\} \le c_1(b-a),$$
then
$$a' d_G \le v \le b' d_G \quad \mbox{in} \quad S_\theta,$$
for some $a'$, $b'$ that satisfy
$$ a \le a' \le b' \le b , \quad \quad b'-a' \le \eta (b-a),$$ with $\eta \in (0,1)$, universal, close to 1.
\end{prop}

\begin{proof}
We define the functions
\begin{equation*}
 v_1=\frac{v- a \,  d_ G }{b - a}, ~
 v_2=\frac{b \,  d_ G-v} {b - a}
 \end{equation*}
which are nonnegative. Since
\begin{equation*}
 v_1 + v_2 = d_G,
\end{equation*}
we might assume (see Remark \ref{bardelta}) that the function $v_1$ satisfies
\begin{equation*}
| \{v_1 \geq \frac{\delta}{2}\} \cap B_{\bar \delta}(2 \delta e_n)| \ge \frac 12 |B_{\bar \delta}(2 \delta e_n) |.
\label{posi-pt}
\end{equation*}
Next we apply Theorem \ref{osci-thm}, for the function 
$$\tilde v _1 :=v_1 + c_1 (k^{-2}-|x|^2).$$
Notice that $\tilde v_1 \ge v_1 \ge 0$ in $\Omega$ and 
$$L_u \tilde v_1 \le (g +  \sigma  \, \,  tr \, U)(b-a)^{-1} - 2 c_1 \, tr U \le 0.$$ 

Using Lemma \ref{uni-convex} we can apply weak Harnack inequality Theorem \ref{osci-thm} a finite number of times and obtain
$$\tilde v_1 \ge 2c_2 >0  \quad \mbox{on} \quad F_\delta,$$
for some universal $c_2$. By choosing $c_1$ sufficiently small we find
$$ v_1 \ge c_2 \quad \mbox{on} \quad F_\delta.$$

Now we can apply Proposition \ref{key-thm} to $v_1/c_2$ since
$$ L_u(v_1/c_2) \le  2(c_1/c_2) \, tr U \le \delta^\frac{1}{n-1} \, tr \, U,$$
provided that $c_1$ is small. We obtain
$$v_1 \ge (c_2/2) \, d_G  \quad \mbox{in} ~S_\theta,$$
hence 
$$ v \ge a' d_G, \quad a'=a+ c_2(b-a)/2.$$

\end{proof}

\section{Proof of Theorem  \ref{main-reg} }\label{proof-sec}

Throughout this section we assume that $u$, $v$ satisfy the hypotheses of 
Theorem \ref{main-reg} and we also assume for simplicity that
 $$u(0)=0, \quad \nabla u(0)=0.$$ 
Our boundary gradient estimate states as follows.
\begin{prop} Let $v$ be as in Theorem \ref{main-reg}. Then, in $\Omega\cap B_{\rho/2}$, we have
\begin{equation*}
 \abs{v(x)} \leq C (\|v\|_{L^{\infty}(\Omega\cap B_{\rho})} +\|g/ \, tr \, U\|_{L^\infty(\Omega\cap B_{\rho})})  d_{\p\Omega}(x).
\end{equation*}
\label{bdr-grad}
\end{prop}
The proposition follows easily from the construction of a suitable supersolution.
\begin{lem}[Supersolution] 
There exists universal constants $M$ large, and $\tilde \delta$ small such that the function
$$\overline w:= M x_{n} +  u - \tilde \delta |x'|^2 - \frac{\Lambda^n}{( \lambda \tilde \delta)^{n-1}} 
 x_{n}^2$$ satisfies 
 $$L_{u} (\overline w) \leq - \tilde \delta\, \, tr \, U,$$
and
$$\overline w \ge 0 ~\text{on}~ \p (\Omega \cap B_\rho), \quad \overline w \ge \tilde \delta ~\text{on}~  \p (\Omega \cap B_\rho) \setminus B_{\rho/2}.$$

\label{key-lem_sup}
\end{lem}

\begin{proof}
We first choose $\tilde \delta \le \rho$ small such that $$u- \tilde \delta |x'|^2 \ge \tilde \delta \quad \mbox{on} ~\overline \Omega \setminus B_{\rho/2}.$$
The existence of $\tilde \delta$ follows for example from \eqref{lbd}.
We choose  $M$ such that
$$ M x_n - \frac{\Lambda^n}{( \lambda \tilde \delta)^{n-1}} 
 x_{n}^2  \ge 0  \quad \quad \text{on}~\overline \Omega.$$
Then on $\p \Omega$, $$\overline w \ge u -  \tilde \delta  |x'|^2 \ge 0,$$
and we obtain the desired inequalities for $\overline w$ on $\p \Omega$.
  
If we denote $$q(x):=\frac 12 \left( \tilde \delta  |x'|^2 + \frac{\Lambda^n}{( \lambda \tilde \delta)^{n-1}} 
 x_{n}^2 \right),$$
then $$\det D^2 q=\frac{\Lambda^n}{\lambda^{n-1}}, \quad D^2q \ge \tilde \delta I$$
and we obtain as in Lemma \ref{key-lem} $$L_u \overline w \le - \tilde \delta \, \,  tr \, U.$$

\end{proof}

\begin{proof}[Proof of Proposition \ref{bdr-grad}]
By dividing the equation by a suitable constant we may suppose that
$$\|v\|_{L^\infty} \le \tilde \delta, \quad \quad \|g/ \, tr \, U\|_{L^\infty} \le  \tilde \delta,$$
and we need to show that $$|v| \le C d_{\p \Omega} \quad \mbox{in}~\p \Omega \cap B_{\rho /2}.$$

Since $v \le \overline w$ on $\p (\Omega \cap B_\rho)$ and $ L_u v \geq L_u \overline w$ we obtain $v \le \overline w$ in $ \Omega \cap B_\rho$ hence $$v(0,x_n) \le Cx_n, \quad \quad\mbox{if} ~ x_n \in [0, \rho /2].$$  

The same argument applies at all points $x_0 \in \p \Omega \cap B_{\rho/2}$ and we obtain
the upper bound for $v$. The lower bound follows similarly and the proposition is proved.
 
 \end{proof}

\begin{proof}[Proof of Theorem \ref{main-reg}]

By dividing by a suitable constant we may suppose that 
$$\|v\|_{L^\infty} + \|g / \,  tr \, U \|_{L^\infty}$$
is sufficiently small such that, by Proposition \ref{bdr-grad},
$$ |v| \le \frac{1}{2}d_{\p \Omega} \quad \mbox{in} ~ \overline \Omega \cap B_{\rho /2}.$$

We focus our attention on the sections at the origin and we show that 
we can improve these bounds in the form 
\begin{equation}
a_{h} d_{\p\Omega} \leq v\leq  b_{h} d_{\p\Omega}, ~ \text{in} ~ S_{h},
\label{Sh-ineq}
\end{equation}
for appropriate constants $a_h$, $b_h$. First we fix $h_0$ small universal and let $$a_{h_0}=-1/2, \quad b_{h_0}=1/2.$$
Then we show by induction that for all $$h=h_0 \theta^k, \quad \quad k \ge 0,$$ we can find $a_h$ increasing  and $b_h$ decreasing with $k$ such that \eqref{Sh-ineq} holds and
\begin{equation}\label{b-a}
b_h-a_h=\left (\frac{1+\eta}{2}\right)^k  \ge C_1 h^{1/2} |\log h|^2.
\end{equation}
for some large universal constant $C_1$.
We notice that this statement holds for $k=0$ if $h_0$ is chosen sufficiently small.

Assume the statement holds for $k$. Proposition \ref{bdr-distort} implies that
\begin{equation*}
 \overline{a}_h d_{\p \Omega_{h}} \leq v_{h}  \leq \overline{b}_h d_{\p \Omega_{h}} \quad  ~\text{in} ~S_{1} (u_{h}) 
\end{equation*}
with
\begin{equation*}
|\overline{a}_h - a_{h}| \le C h^{1/2}\abs{\log h}^{2},~
 |\overline{b}_h-b| \le C h^{1/2}\abs{\log h}^{2}.
\end{equation*}
Since $$(u_h, S_1(u_h) ) \in \mathcal D_\sigma, \quad \text{for}~ \sigma=C h^{1/2},  $$
and (see \eqref{sol-trans}, \eqref{gh})
$$L_{u_h}v_h=g_h,$$ 
 \begin{align*}
 \|g_h/tr U_h\|_{L^\infty} &\le C  h^{1/2} |\log h|^2 \|g/tr U\|_{L^\infty}\\
 & \le Ch^{1/2} |\log h |^2\\
 &\le c_1(\overline b_h -\overline a_h),
 \end{align*}
we can apply Proposition \ref{iteration} and conclude
\begin{equation*}
 \overline{a}_{\theta h} d_{\p \Omega_{h}} \leq v_{h} (x)
  \leq \overline b_{\theta h} d_{\p \Omega_{h}}, ~\text{in}~ S_{\theta} (u_{h}). 
\end{equation*}
with
\begin{equation*}
 \overline b_{\theta h}-  \overline a_{\theta h} \leq \eta ( \overline{b} _h- \overline{a}_h).
\end{equation*}
Rescaling back to $S_{\theta h}$, and using Proposition \ref{bdr-distort} again, we obtain
\begin{equation}
 a_{\theta h} d_{\p\Omega} \leq v \leq  b_{\theta h} d_{\p\Omega}, ~  \text{in}~ S_{\theta h}(u)
\label{Stheta-ineq}
\end{equation}
where
  $$b_{\theta h}-a_{\theta h} \le \eta (b_h-a_h) + C h^{1/2} |\log h| \le (1+\eta)/2 (b_h-a_h)  .$$
By possibly modifying their values we may take $a_{\theta h}$, $b_{\theta h}$ such that
$$ a_h \le a_{\theta h} \le b_{\theta h} \le b, \quad \quad b_{\theta h}-a_{\theta h} =  \frac{1+\eta}{2} (b_h-a_h).$$

From \eqref{Sh-ineq}, \eqref{b-a} we find
$$osc_{S_h} v \le C h^{1/2 + \alpha}, $$
for some small $\alpha$ universal. Using \eqref{small-sec} we obtain
$$osc_{B_r} v \le C r^{1+\alpha} \quad \mbox{if}~r \le c,$$
and the theorem is proved.
\end{proof}

\section{Proof of Theorem \ref{main-reg-gl2}}

In this last section we prove Propositions \ref{u_reg} and \ref{tan_sec} and Theorem \ref{main-reg-gl2}.

\begin{proof}[Proof of Proposition \ref{u_reg}]
Let $y\in \Omega $ with $$r:=d_{\p \Omega}(y) \le c,$$ and consider the maximal section $S_{\bar{h},y}$ centered at $y$, i.e.,
$$\bar{h}=max\{h\,| \quad S_{y,h}\subset \Omega\}.$$
By Proposition \ref{tan_sec} applied at the point $$x_0\in \p S_{y,\bar h} \cap \p \Omega,$$ we have
 $$\bar h^{1/2} \sim r, \quad |\nabla u(y)-\nabla u(x_0)| \le C \bar h^{1/2},$$
and $S_{\bar h,y}$ is equivalent to an ellipsoid $E$ i.e
$$c\tilde E \subset S_{\bar h,y}-y \subset C\tilde E,$$
where
$$E :=\bar h^{1/2}A_{\bar{h}}^{-1}B_1, \quad \mbox{with} \quad \|A_{\bar{h}}\|, \|A_{\bar h}^{-1} \| \le C |\log \bar h|.$$
We denote $$u_y:=u-u(y)-\nabla u(y) (x-y).$$
The rescaling $\tilde u: \tilde S_1 \to \R$ of $u$ 
$$\tilde u(\tilde x):=\frac {1}{ \bar h} u_y(T \tilde x) \quad \quad x=T\tilde x:=y+\bar h^{1/2}A_{\bar{h}}^{-1}\tilde x,$$
satisfies
$$\det D^2\tilde u(\tilde x)=\tilde f(\tilde x):=f(T \tilde x),  $$
and
$$B_c \subset \tilde S_1 \subset B_C, \quad \quad \tilde S_1=\bar h^{-1/2} A_{\bar h}(S_{\bar h,y}- y),$$
where $\tilde S_1$ represents the section of $\tilde u$ at the origin at height 1.

The interior $C^{1,\gamma}$ estimate for solutions of the Monge-Ampere equation (see \cite{C1}) gives
$$|\nabla \tilde u(\tilde z_1)-\nabla \tilde u(\tilde z_2)| \le C |\tilde z_1-\tilde z_2|^{\gamma} \quad \forall \tilde z_1,\tilde z_2 \in \tilde S_{1/2}  $$
for some $\gamma \in (0,1)$, $C$ universal. Rescaling back and using
$$\nabla \tilde u(\tilde z_1)-\nabla \tilde u(\tilde z_2)=\bar h^{-1/2} (A_{\bar h}^{-1})^T (\nabla u( z_1)-\nabla  u(z_2)), \quad \quad \tilde z_1-\tilde z_2=\bar h^{-1/2}A_{\bar h}(z_1-z_2)$$
we find
$$|\nabla u(z_1)-\nabla u( z_2)|  \le  |z_1-z_2|^\gamma \quad \forall  z_1, z_2 \in  S_{\bar h/2,y} .$$
Notice that this inequality holds also in the Euclidean ball $B_{r^2}(y)\subset S_{\bar h/2,y}$. Also, if  $y_0 \in \p \Omega$ denotes the closest point to $y$ on $\p \Omega$ i.e $|y-y_0|=r$, by  Lemma \ref{U-bound-grad2}, we find
$$|\nabla u(y)-\nabla u(y_0)| \le |\nabla u(y)-\nabla u(x_0)|+|\nabla u(x_0)-\nabla u(y_0)| \le r^{1/2}.   $$ 
These oscillation properties for $\nabla u$ and Lemma \ref{U-bound-grad2} easily imply that $$[\nabla u]_{C^\alpha(\bar \Omega)} \le C,$$ for some $\alpha\in (0,1)$, $C$ universal.

If we assume that $f \in C^\beta (\overline \Omega)$ then $$\|\tilde f\|_{C^{\beta} (\tilde{S}_{1})}\le \|f\|_{C^\beta (\overline \Omega)},$$
and the interior $C^{2,\beta}$ estimates  
for $\tilde{u}$ in $\tilde{S}_{1}$ (see \cite{C2}) give 
\begin{equation}\label{d2_bd}
\|D^2 \tilde u\|_{C^\beta(\tilde S_{1/2})} \le K.
\end{equation}
 In particular $$\|D^2 u(y)\|=\|A_{\bar h}^T D^2 \tilde u(0) A_{\bar h}  \| \le K |\log h|^2 \le K |\log r|^2,$$
 where by $K$ we denote various constants depending on $\beta$, $\|f\|_{C^\beta (\overline \Omega)} $ and the universal constants.
 
 \end{proof}
 
 \begin{proof}[Proof of Theorem \ref{main-reg-gl2}]
 
 We use the same notations as in the proof of Proposition \ref{u_reg}. 
 
 After multiplying $v$ by a suitable constant we may assume that $$\|\varphi\|_{C^{1,1}}+\|g/ \, tr \, U\|_{L^\infty}=1.$$
 We define also the rescaling $\tilde v$ for $v$
 $$\tilde v(x):=\bar h^{-1/2} v(Tx).$$
  From Theorem \ref{main-reg-gl} we obtain $$\max_{S_{\bar h, y}} | v-v(x_0)-\nabla v(x_0)(x-x_0)| \le C \, r^{1+\alpha'}$$
 for some universal $\alpha'\in (0,1)$ and $C$, hence $$\max_{\tilde S_1} |\tilde v(\tilde x) -\tilde v(\tilde x_0)-\nabla \tilde v(\tilde x_0)(\tilde x -\tilde x_0)| \le C r^{1+\alpha '}\bar h^{-1/2} \le C r^{\alpha '} .$$
 Using the computations in Section 4 we see that $\tilde v$ solves
 $$\tilde U^{ij}\tilde v_{ij}= \tilde g(x):=\bar h^{1/2} g(Tx) ,$$
 with
 $$ \| \tilde g(x)/ tr \, \tilde U \|_{L^\infty(\tilde S_{1/2})} \le C h^{1/2}|\log h|^2 \|g/ \, tr \, U\|_{L^\infty}\le C r^{\alpha '}.$$

 Since \eqref{d2_bd} holds, we can apply Schauder estimates and find that for any $\tilde z_1, \tilde z_2\in \tilde S_{1/4}$
 $$|\nabla \tilde v(0)-\nabla \tilde v(\tilde x_0)|\le K r^{\alpha'},\quad \quad |\nabla \tilde v(\tilde z_1)-\nabla \tilde v(\tilde z_2)| \le K r^{\alpha'}|\tilde z_1-\tilde z_2|^{\alpha'/2} .$$
 Using that $\nabla \tilde v(\tilde z_i)=(A_{\bar h}^{-1})^T \nabla v(z_i)$ 
 we obtain
 $$ |\nabla v(y)-\nabla v(x_0)|\le K r^{\alpha'/2}, \quad |\nabla v(z_1)-\nabla v(z_2)|\le K|z_1-z_2|^{\alpha'/2}$$ for any $z_1,z_2\in B_{r^2}(y)$. These inequalities and Theorem \ref{main-reg-gl}  give as in the proof of Proposition \ref{u_reg} above the desired $C^{1,\alpha}$ bound for $v$.
  
 \end{proof}

\begin{rem}
Theorem \ref{main-reg-gl2} still holds if we only assume that $f \in C(\overline \Omega)$. In this case one needs to apply the interior $C^{1,\alpha}$ estimates for the linearized Monge-Ampere equation obtained by Guti\'errez and Nguyen  \cite{GN}. 
\end{rem}

We conclude the paper with a sketch of the proof of Proposition \ref{tan_sec}.

\begin{proof}[Proof of Proposition \ref{tan_sec}]
Assume that the hypotheses of the Localization Theorem \ref{main_loc} hold at the origin. For $a\ge 0$ we denote
$$S_a':=\{ x \in \overline \Omega| \quad u(x)<ax_n\} ,$$
and clearly $S_{a_1}'\subset S_{a_2}'$ if $a_1 \le a_2$.
The proposition easily follows once we show that $S_{ch^{1/2}}'$ has the shape of the ellipsoid $E_h$ for all small $h$.

From Theorem \ref{main_loc} we know $$S_h:=\{u<h\} \subset k^{-1} E_h \subset \{x_n \le k^{-1} h^{1/2} \} $$ and since $u(0)=0$ we use the convexity of $u$ and obtain
\begin{equation}\label{f_sub}
S_{kh^{1/2}}' \subset S_h \cap \Omega.
\end{equation}
This inclusion shows that in order to prove that $S_{kh^{1/2}}'$ is equivalent to $E_h$ it suffices to bound its volume by below
$$|S_{kh^{1/2}}'| \ge c|E_h|.$$

From Theorem \ref{main_loc}, there exists $y \in \partial S_{\theta h}$ such that $y_n \ge k(\theta h)^{1/2}$. 
We evaluate $$\tilde u:=u-k h^{1/2}x_n, $$ at $y$ and find $$\tilde u(y) \le \theta h - k h^{1/2} k (\theta h)^{1/2} \le -\delta h,$$ for some $\delta>0$ provided that we choose $\theta$ small depending on $k$. 
Since $\tilde u=0$ on $\p S_{kh^{1/2}}'$ and $$ \det D^2 \tilde u \ge \lambda$$ we have 
$$|\inf \tilde u| \le C(\lambda)|S_{kh^{1/2}}'|^{2/n},$$ hence $$ c h^{n/2} \le |S_{kh^{1/2}}'|.$$

\end{proof}

\end{document}